\documentclass[a4paper,11pt,intlimits,reqno]{amsart}
\usepackage[usenames]{color}
\usepackage[colorlinks=true, linkcolor=webgreen, filecolor=webbrown, citecolor=webgreen]{hyperref}
\usepackage[T1]{fontenc}
\usepackage{lmodern}
\usepackage{tabularx}
\usepackage{multirow}
\usepackage{centernot}
\usepackage[english]{babel}
\usepackage{amsmath}
\usepackage{amsfonts}
\usepackage{tikz}
\usepackage{amssymb}
\usepackage{mathrsfs}
\usepackage[babel]{microtype}
\usepackage{url}
\usepackage{amsthm}
\usepackage{tabularx}
\usepackage{multirow}
\usepackage{changepage}

\newtheorem{theorem}{Theorem}[section] \newtheorem{lemma}{Lemma}[theorem]
\newtheorem{proposition}[theorem]{Proposition}

\newtheorem{remark}{Remark}

\newtheorem{corollary}[theorem]{Corollary}

\newtheorem{definition}[theorem]{Definition}

\definecolor{webgreen}{rgb}{0,.5,0}
\definecolor{webbrown}{rgb}{.6,0,0}

\newfont{\cyrr}{wncyr10}

\definecolor{lightgrey}{rgb}{0.8, 0.84, 0.8}

\DeclareMathOperator{\tr}{tr}

\newcommand{\Fq}{{\mathbb{F}}_{q}}

\everymath{\displaystyle}

\begin{document}

\title[Moments of Gaussian hypergeometric functions]
{Moments  of Gaussian hypergeometric functions over finite fields}

\author[P\lowercase{al}, R\lowercase{oy}, S\lowercase{adek}]{A\lowercase{nkan} P\lowercase{al}, B\lowercase{idisha} R\lowercase{oy}, \lowercase{and} M\lowercase{ohammad} S\lowercase{adek}}

\address{Ankan Pal, Department of Mathematics, University of L'Aquila, Via Vetoio, L'Aquila - 67100, Italy.}
\email{ankanpal100@gmail.com}

\address{Bidisha Roy, Institute of Mathematics of the Polish Academy of Sciences, Jana i J{\c e}drzeja {\' S}niadeckich 8, Warsaw 00-656, Poland}
\email{brroy123456@gmail.com}

\address{Mohammad Sadek, Faculty of Engineering and Natural Sciences,
Sabanc{\i} University,
Tuzla, \.{I}stanbul, 34956 Turkey}
\email{mohammad.sadek@sabanciuniv.edu}

\begin{abstract}
We prove explicit formulas for certain first and second moment sums of families of Gaussian hypergeometric functions $_{n+1}F_n$, $n\ge1$, over finite fields with $q$ elements where $q$ is an odd prime. This enables us to find an estimate for the value $_6F_5(1)$. In addition, we evaluate certain second moments of traces of the family of Clausen elliptic curves in terms of the value $_3F_2(-1)$.   
These formulas also allow us to express the product of certain $_2F_1$ and $_{n+1}F_n$ functions in terms of finite field Appell series which generalizes current formulas for products of $_2F_1$ functions.  We finally give closed form expressions for sums of Gaussian hypergeometric functions defined using different multiplicative characters.
 \end{abstract}

\maketitle
\let\thefootnote\relax
\let\thefootnote\relax

\section{Introduction}\label{sec:Intro}
\label{sec:introduction}

In 1987, Greene introduced {\it Gaussian hypergeometric functions over finite fields} as analogues of the classical hypergeometric functions, \cite{Bailey}. These Gaussian hypergeometric functions garnered much attention since then. Not only they share arithmetic properties with their calssical counterparts, they have also turned out to possess many connections with several arithmetic-geometric objects.  Values of Gaussian hypergeometric functions are expressed in terms of traces of families of elliptic curves, \cite{Kalita,barmanPAMS,Guindy+Ono,Ono,SadN,Sad19}. The number of rational points on certain hyperelliptic curves were linked to values of some Gaussian hypergeometric functions, \cite{Saikia,Sad}. These functions also appeared in connection with studying the arithmetic of certain $K3$ surfaces and Calabi-Yau threefolds, \cite{Ono1,Ah}.

Let $\epsilon$ and $\phi$ be the trivial and quadratic character, respectively, over a finite field $k$. In \cite{Ono+Saad+Saikia}, the asymptotic behaviour of  {\it moments} of certain Gaussian hypergeometric function was studied. Namely, the sum
$ \displaystyle \sum_{ x \in k } {}_{n+1} F_n\left(\begin{matrix} \phi & \phi \cdots  \phi \\ & \epsilon \cdots \epsilon \end{matrix} \ \middle| \  x \right)^m $, $n=1,2$, is evaluated; and the limiting distribution is proved to be semicircular when $n=1$, and it is a Batman distribution for the traces of the real orthogonal group $O_3$ when $n=2$, as $ |k| \rightarrow + \infty $, for any positive integer $m$. 

Let $\Fq$ be the finite field with $q$ elements where $q$ is an odd prime. In this work, we are concerned with computing moments of the following format 
\begin{eqnarray}
\label{eqq}
\sum_ { x \in \mathbb{F}_ q} A(x) \cdot {}_{n+1} F_n\left(\begin{matrix} \phi & \phi \cdots  \phi \\ & \epsilon \cdots \epsilon \end{matrix} \ \middle| \  x \right)^m \ 
\end{eqnarray}
where $A=\epsilon$ or $\phi$, and $ m=1,2$.

When $m =1$,  the sum above is evaluated explicitly in \S \ref{sec:Moments}.  As a consequence, we are able to present second moments expressions for the family of Clausen elliptic curves, namely 
\begin{eqnarray*}
\sum_{\lambda\ne0,-1}\phi(1+\lambda) a'_{\lambda}(q)^2+q + q^2\,\, _3F_2\left( \begin{matrix} \phi & \phi & \phi \\
 & \epsilon & \epsilon  \end{matrix} \middle| 1\right)&=&-1,\\
\sum_{\lambda\ne0,-1} \phi(\lambda) a'_{\lambda}(q)^2+q\,\phi(-1) + q^2\,\, _3F_2\left( \begin{matrix} \phi & \phi & \phi \\
 & \epsilon & \epsilon  \end{matrix} \middle| 1 \right)&=&-\phi(-1),
 \end{eqnarray*}
 where $ a^\prime_{\lambda} (q)$ is the trace of Frobenius of the Clausen elliptic curve $y^2=(x-1)(x^2+\lambda)$, $\lambda\ne-1,0$.  In particular, it can be seen that both moments can be evaluated in terms of the value $_3F_2(-1)$ which can be found as \cite[Theorem 6]{Ono}.
 
 We obtain a closed form expression for the sum (\ref{eqq}) when $m=2$ and $A=\phi$, namely, we prove that for any integer $n\ge 2$, the following identity holds
 $$\sum_{ \lambda \in \mathbb{F}_q} \phi(\lambda) \cdot {}_{n} F_{n-1} \left(\begin{matrix} \phi & \phi \cdots  \phi \\ & \epsilon \cdots \epsilon \end{matrix} \ \middle| \  \lambda \right)^2=q\, \phi(-1)^k \cdot {}_{2n} F_{2n-1} \left(\begin{matrix} \phi & \phi \cdots  \phi \\ & \epsilon \cdots \epsilon \end{matrix} \ \middle| \  1 \right).$$  In order to achieve this goal, we derive an inductive representation of Gaussian hypergeometric functions. Another application of the latter inductive representation is finding a new product formula for hypergeometric functions of type $_2F_1$ and $_{n+1}F_n$, $n\ge2$,
in terms of finite field analogues of Appell series. This generalizes product formulas that were obtained recently for Gaussian hypergeometric functions of type $_2F_1$ in \cite{Barman+Tripathi-1}.

Although special values of the Gaussian hypergeometric function ${}_2 F_1 \left(\begin{matrix}
\phi & \phi \\
 & \epsilon 
\end{matrix} \middle| x \right)$ and ${}_3 F_2 \left(\begin{matrix}
\phi & \phi & \phi \\
 & \epsilon & \epsilon
\end{matrix} \middle| x \right)$ have been found for some values of $x$, \cite{Ono}, the literature lacks general strategies to evaluate these functions at any $x$.  Therefore, estimates of certain Gaussian hypergeometric functions have also received some attention. In \cite{Ono1}, the value   $_4F_3(1):= {}_4F_3\left( \begin{matrix} \phi & \phi & \phi & \phi \\
& \epsilon & \epsilon & \epsilon  \end{matrix} \middle| 1\right)$ is proved to satisfy $|_4F_3(1)+1/q^2|\le 2/q^{3/2}$. This was realized by observing that the Calabi-Yau threefold $x+1/x+y+1/y+z+1/z+w+1/w=0$ is modular. By exploiting a second moment sum $\sum_x\phi(x)\,\,{}_3 F_2 \left(\begin{matrix}
\phi & \phi & \phi \\
 & \epsilon & \epsilon
\end{matrix} \middle| x \right)^2$, we prove that $|_6F_5(1)|=O(1/q^2)$.

We furthermore produce a generating function for Gaussian hypergeometric function of type $_{n+1}F_n$. Consequently, we find closed forms for sums of the form   
$
\setlength\arraycolsep{1pt} 
\sum_{\psi\ne \epsilon} {}_{n+2 } F_{n+1} \left(\begin{matrix} A & \ \ A & \ \ \cdots & \ \ A& \ \ \psi \\ & \ \ \epsilon& \ \ \cdots & \ \ \epsilon & \ \ \epsilon\end{matrix} \ \middle| x \right)\psi(t)$ where $A$ is a multiplicative character over $\Fq$, and the sum is over all nontrivial multiplicative characters over $\Fq$. In fact, the latter sum is expressed in terms of hypergeometric functions of type $_{n+1}F_n$.  In particular, we compute the latter sum for certain values of $t\in\Fq-\{0,1\}$ when $n$ is either $1$ or $2$ in terms of the traces of the Frobenius of Legendre and Clausen elliptic curves, respectively.

\section{Gaussian hypergeometric functions over finite fields}
\label{sec:pre}
In this section, we recall the main definitions and results that we will use throughout the paper. The finite field with $q$ elements will be denoted $\mathbb{F}_q$, where $q$ is an odd prime. We write $\epsilon$ and $\phi$ for the trivial and quadratic characters over $\Fq$ respectively. If $A$ and $B$ are multiplicative characters over $\Fq$, then their {\em Jacobi sum}, $J(A,B)$, is defined as 
$$J(A,B)=\sum_{x\in\Fq}A(x)B(1-x).$$ It is clear that $J(A,B)=J(B,A)$.  The {\em binomial coefficient} $\begin{pmatrix}
A \\ B
\end{pmatrix}$ is defined by $$\begin{pmatrix}
A \\ B
\end{pmatrix}= \frac{B(-1)}{q}J(A,\overline{B})$$ where $\overline B$ is the character inverse of $B$.  Greene, \cite{Greene}, defined the \textit{Gaussian hypergeometric series} for the multiplicative characters $A_i$, $i=0,1,\cdots,n$, and $B_i$, $i=1,\cdots,n,$ and $x\in\Fq$, in the following way \begin{equation}\label{eq:defnhypergeometric}
\setlength\arraycolsep{1pt}
{}_{n + 1} F_n \left(\begin{matrix} A_0 & \ \ A_1 & \ \ \cdots\ \ & A_n \\ & \ \ B_1 & \ \ \cdots\ \ & B_n \end{matrix} \ \middle| \ x \right) = \frac{q}{q - 1} \sum\limits_{\chi} \begin{pmatrix}
A_0 \chi \\ \chi 
\end{pmatrix} \begin{pmatrix}
A_1 \chi \\ B_1 \chi 
\end{pmatrix} \cdots \begin{pmatrix}
A_n \chi \\ B_n \chi
\end{pmatrix} \chi(x),\qquad n\ge1,
\end{equation} 
where $_1F_0(A|x)=\epsilon(x)\overline{A}(1-x).$

The following inductive representation of Gaussian hypergeometric series is \cite[Theorem 3.13]{Greene}.
\begin{theorem}\label{th:GreensTh3.13} For characters $A_0, A_1,\cdots, A_n$ and $B_1,\cdots,B_n$ over $\Fq$ and $x \in \Fq$ we have,
\begin{equation*}
\setlength\arraycolsep{1pt} 
{}_{n + 1} F_n \left(\begin{matrix} A_0 & \ \ A_1 & \ \ \cdots & \ \ A_n \\ & \ \ B_1 & \ \ \cdots & \ \ B_n \end{matrix} \ \middle| \ x \right) = \frac{A_n B_n (-1)}{q} \sum\limits_{y} \setlength\arraycolsep{1pt} {}_n F_{n - 1} \left(\begin{matrix} A_0 & \ \ A_1 & \ \ \cdots & \ \ A_{n - 1} \\ & \ \ B_1 & \ \ \cdots & \ \ B_{n - 1} \end{matrix} \ \middle| \ xy \right) A_n(y) {\overline{A}}_n B_n (1 - y).
\end{equation*}
\end{theorem}

Let $ \zeta_p$ be a fixed primitive $p$-th root of unity in $\mathbb{C}$. The trace map $ \tr : \mathbb{F}_ q \rightarrow \mathbb{F}_p$ is given by $ \tr( \alpha) = \alpha + \alpha^{p} + \alpha^{p^2}+\cdots + \alpha^{ p^{r -1}}$. The additive character $ \theta: \mathbb{F}_q \rightarrow \mathbb{C}$ is defined by $ \theta(\alpha) = \zeta_p^{ \tr(\alpha)}$. For any multiplicative character $ \chi $ over $ \mathbb{F}_q$, the {\it Gauss sum} is defined by $$ g (\chi)= \sum_{ x \in \mathbb{F}_q} \chi (x) \theta(x).$$ 
\begin{definition}
For any multiplicative characters $ A, B, C, C^\prime$ over $ \mathbb{F}_q$ and for any $ x, y \in \mathbb{F}_q$, the {\em finite field Appell series} is defined by
$$ F_4 (A; B; C, C^\prime ; x, y )^* = \frac{1}{(q-1)^2} \sum_{ \chi, \lambda} \frac{ g ( A \chi \lambda) g ( B \chi \lambda) g (\overline{C \chi}) g ( \overline{C^ \prime \lambda}) g ( \overline{\lambda}) g(\overline{\chi})}{g(A) g(B) g (\overline{C}) g (\overline{C^\prime})} \chi(x) \lambda(y)$$
where the sum is over multiplicative characters $\chi$ and $\lambda$ over $\Fq$.
\end{definition}
The following lemma from \cite{Barman+Tripathi-1} provides a relation between finite field Appell series and the product of two Gaussian hypergeometric series.
\begin{lemma}\label{lem:BarmanTripathiLemma3.2} \cite[Lemma 3.2]{Barman+Tripathi-1} Let $A, B$ be nontrivial multiplicative characters different from the character $C$ over $\Fq$, and $z,w\in\Fq-\{1\}$. Then one has 
\begin{align*}
 {}_2 F_1 \left(\begin{matrix} A & B \\ & C \end{matrix} \ \middle| \ z \right) \ \ {}_2 F_1 \left(\begin{matrix} A & B \\ & AB\overline{C} \end{matrix} \ \middle| \ w \right)& = \frac{A(-1) g(B) g(\overline{C}) g(\overline{AB}C)}{q g(\overline{B}) g(B\overline{C}) g(\overline{A}C)} F_4 (A;B;C;AB\overline{C};z(1-w);w(1 - z))^* \\ 
&+  \frac{q B(-1) \overline{A}(1 - z) \overline{B}C(w) \overline{C}(1 - w)}{g(A) g(\overline{B}) g(B\overline{C}) g(\overline{A}C)} \delta\left(\frac{1 - w - z}{(1 - z)(1 - w)}\right), 
\end{align*}
where $\delta(x) = 1$ if $x = 0$, and $\delta(x) = 0$ otherwise.
\end{lemma}

\section{An inductive representation}
\label{sec:Induction}

In what follows $q$ is an odd prime. 
In this section we produce an inductive representation of Gaussian hypergeometric functions over finite fields.

\begin{theorem}\label{thm:general_k}
Let $ n$ and $k$ be integers with $ n > k \geq 1$. For multiplicative characters $ A_0, A_1, \ldots A_{n-k}$ and $ B_1, \ldots, B_{n-k}$ over $ \Fq$ and $ x \in \Fq$, we have 
{\footnotesize\begin{align*}
 &{}_{n + 1} F_n \left(\begin{matrix} A_0 & \ \ A_1 &  \cdots & \ \ A_{n-k} &\phi &  \cdots  \phi\\& \ \ B_1 &  \cdots & B_{n-k}& \epsilon& \cdots \ \epsilon \end{matrix} \ \middle| \ x \right) \\&= 
 \frac{ \phi(-1)^k}{q} \sum_{b\in\Fq}\phi(b)\,\setlength\arraycolsep{1.5pt}{}_{n+1-k} F_{n - k} \left(\begin{matrix} A_0 & \ \ A_1 & \ \ \cdots & \ \ A_{n - k} \\ & \ \ B_1 & \ \ \cdots & \ \ B_{n - k} \end{matrix} \ \middle| \ bx \right){}_{k} F_{k-1} \left(\begin{matrix} \phi & \ \ \phi \cdots & \ \ \phi \\ & \ \ \epsilon \cdots & \ \ \epsilon  \end{matrix} \ \middle| \ b\right).
\end{align*}}
\end{theorem}

\begin{proof}
We use the inductive representation in Theorem \ref{th:GreensTh3.13}  to obtain the following identities 
\begin{footnotesize}
\begin{align*}
\setlength\arraycolsep{1pt} 
{}_{n + 1} F_n \left(\begin{matrix} A_0 & \ \ A_1 &  \cdots & \ \ A_{n}\\& \ \ B_1 &  \cdots & \ \ B_{n} \end{matrix} \ \middle| \ x \right)= \frac{A_n B_n (-1)}{q} \sum\limits_{y} \setlength\arraycolsep{1pt} {}_n F_{n - 1} \left(\begin{matrix} A_0 & \ \ A_1 & \ \ \cdots & \ \ A_{n - 1} \\ & \ \ B_1 & \ \ \cdots & \ \ B_{n - 1} \end{matrix} \ \middle| \ xy \right) A_n(y) {\overline{A}}_n B_n (1 - y)\\
=
\frac{A_n A_{n-1}B_nB_{n-1} (-1)}{q^2} \sum\limits_{y} A_n(y) {\overline{A}}_n B_n (1 - y)\setlength\arraycolsep{1pt} \sum_z{}_{n-1} F_{n - 2} \left(\begin{matrix} A_0 & \ \ A_1 & \ \ \cdots & \ \ A_{n - 2} \\ & \ \ B_1 & \ \ \cdots & \ \ B_{n - 2} \end{matrix} \ \middle| \ xyz \right) A_{n-1}(z) {\overline{A}}_{n-1} B_{n-1} (1 - z)\\
=
\frac{A_nA_{n-1} B_nB_{n-1} (-1)}{q^2} \sum\limits_{y,z} A_n(y) {\overline{A}}_n B_n (1 - y)\setlength\arraycolsep{1pt} {}_{n-1} F_{n - 2} \left(\begin{matrix} A_0 & \ \ A_1 & \ \ \cdots & \ \ A_{n - 2} \\ & \ \ B_1 & \ \ \cdots & \ \ B_{n - 2} \end{matrix} \ \middle| \ xyz \right) A_{n-1}(z) {\overline{A}}_{n-1} B_{n-1} (1 - z).
\end{align*}
\end{footnotesize} We first establish the correction of the inductive identity in the statement of the theorem when $ k =2$:
\begin{footnotesize}
\begin{align*}
\setlength\arraycolsep{1pt} 
{}_{n + 1} F_n \left(\begin{matrix} A_0 & \ \ A_1 &  \cdots & A_{n-2}&\ \ \phi &\ \ \phi\\& \ \ B_1 &  \cdots & \ \ B_{n-2} & \epsilon & \ \ \epsilon\end{matrix} \ \middle| \ x \right)=
 \frac{1}{q^2} \sum\limits_{y,z} \phi(yz)\phi((1 - y)(1-z))\setlength\arraycolsep{1pt} {}_{n-1} F_{n - 2} \left(\begin{matrix} A_0 & \ \ A_1 & \ \ \cdots & \ \ A_{n - 2} \\ & \ \ B_1 & \ \ \cdots & \ \ B_{n - 2} \end{matrix} \ \middle| \ xyz \right) \\
 = \frac{1}{q^2} \sum_{a\in\Fq^{\times}}\sum\limits_{yz=a} \phi(yz) \phi((1 - y)(1-z))\setlength\arraycolsep{1pt} {}_{n-1} F_{n - 2} \left(\begin{matrix} A_0 & \ \ A_1 & \ \ \cdots & \ \ A_{n - 2} \\ & \ \ B_1 & \ \ \cdots & \ \ B_{n - 2} \end{matrix} \ \middle| \ xyz \right)\\
 =\frac{1}{q^2} \sum_{a\in\Fq^{\times}}\sum\limits_{y\in\Fq^{\times}} \phi(a) \phi((1 - y)(1-ay^{-1}))\setlength\arraycolsep{1pt} {}_{n-1} F_{n - 2} \left(\begin{matrix} A_0 & \ \ A_1 & \ \ \cdots & \ \ A_{n - 2} \\ & \ \ B_1 & \ \ \cdots & \ \ B_{n - 2} \end{matrix} \ \middle| \ a x \right)\\
 =\frac{\phi(-1)}{q^2} \sum_{a\in\Fq^{\times}}\phi(a)\setlength\arraycolsep{1.5pt} {}_{n-1} F_{n - 2} \left(\begin{matrix} A_0 & \ \ A_1 & \ \ \cdots & \ \ A_{n - 2} \\ & \ \ B_1 & \ \ \cdots & \ \ B_{n - 2} \end{matrix} \ \middle| \ a x \right)\sum\limits_{y\in\Fq^{\times}}  \phi (y(y-1)(y-a))\\
 =\frac{1}{q} \sum_{a\in\Fq^{\times}}\phi(a)\setlength\arraycolsep{1.5pt} {}_{n-1} F_{n - 2} \left(\begin{matrix} A_0 & \ \ A_1 & \ \ \cdots & \ \ A_{n - 2} \\ & \ \ B_1 & \ \ \cdots & \ \ B_{n - 2} \end{matrix} \ \middle| \ a x \right){}_{2} F_1 \left(\begin{matrix} \phi & \ \ \phi \\& \ \ \epsilon  \end{matrix} \ \middle| \ a\right)
\end{align*}
\end{footnotesize}
where the last equality follows from \cite[Theorem 1]{Ono} together with the fact that $_2F_1(1)=-\phi(-1)/q$.

Suppose the statement of the theorem holds true for any $ \ell \leq k-1$. We now use the induction hypothesis. 

\vspace*{-5mm}

\begin{footnotesize}
\begin{align*}	
& {}_{n + 1} F_n \left(\begin{matrix} A_0 & \ \ A_1 &  \cdots & \ \ A_{n-k} &\phi &  \cdots  \phi\\& \ \ B_1 &  \cdots & B_{n-k}& \epsilon& \cdots \ \epsilon \end{matrix} \ \middle| \ x \right)  \\
&=\frac{\phi(-1)^{k-1}}{q} \sum_{a\in\Fq^{\times}}\phi(a) \setlength\arraycolsep{1pt} {}_{n+1-k+1} F_{n - k+1} \left(\begin{matrix} A_0 & \ \ A_1 & \ \ \cdots & \ \ A_{n - k} & \ \ \phi \\ & \ \ B_1 & \ \ \cdots & \ \ B_{n - k} & \ \ \epsilon \end{matrix} \ \middle| \ ax \right) {}_{k-1} F_{k-2} \left(\begin{matrix} \phi & \ \ \cdots  \phi \\& \ \ \epsilon \cdots \epsilon  \end{matrix} \ \middle| \ a\right)  \\	
&= \frac{\phi(-1)^{k-1}}{q} \sum\limits_{a\in\Fq^{\times}} \phi(a) \frac{\phi \epsilon (-1)}{q} \sum\limits_{y} \phi(y) \phi(1 - y) \setlength\arraycolsep{1pt} {}_{n+1-k} F_{n - k} \left(\begin{matrix} A_0 & \ \ A_1 & \ \ \cdots & \ \ A_{n - k} \\ & \ \ B_1 & \ \ \cdots & \ \ B_{n - k} \end{matrix} \ \middle| \ axy \right) {}_{k-1} F_{k-2} \left(\begin{matrix} \phi & \ \ \cdots  \phi \\& \ \ \epsilon \cdots \epsilon  \end{matrix} \ \middle| \ a\right) \\
&=\frac{\phi(-1)^k}{q^2}\sum\limits_{b \in\Fq^{\times}} \sum\limits_{ay=b }  \phi(b) \phi(1 - y) \setlength\arraycolsep{1pt} {}_{n+1-k} F_{n - k} \left(\begin{matrix} A_0 & \ \ A_1 & \ \ \cdots & \ \ A_{n - k} \\ & \ \ B_1 & \ \ \cdots & \ \ B_{n - k} \end{matrix} \ \middle| \ axy \right) {}_{k-1} F_{k-2} \left(\begin{matrix} \phi & \ \ \cdots \phi \\& \ \ \epsilon  \cdots \epsilon \end{matrix} \ \middle| a\right) \\	
&=\frac{\phi(-1)^k}{q^2} \sum\limits_{b} \phi(b)  \setlength\arraycolsep{1.5pt} {}_{n+1-k} F_{n - k} \left(\begin{matrix} A_0 & \ \ A_1 & \ \ \cdots & \ \ A_{n - k} \\ & \ \ B_1 & \ \ \cdots & \ \ B_{n - k} \end{matrix} \ \middle| \ bx \right) \\
& \times\sum\limits_{y} \phi(1 - y)  \frac{{\phi}^{k - 2} (-1)}{q^{k - 2}} \sum\limits_{x_1, x_2, \ldots, x_{k-2}} \phi(x_1) \phi(x_2)\cdots \phi(x_{k-2}) {\overline{\phi}} \epsilon (1 - x_1)\cdots {\overline{\phi}} \epsilon (1 - x_{k-2}) \overline{ \phi }( 1- by^{-1}x_1 \cdots x_{k-2})\\
&=\frac{\phi(-1)^k}{q} \sum\limits_{b} \phi(b)  \setlength\arraycolsep{1.5pt} {}_{n+1-k} F_{n - k} \left(\begin{matrix} A_0 & \ \ A_1 & \ \ \cdots & \ \ A_{n - k} \\ & \ \ B_1 & \ \ \cdots & \ \ B_{n - k} \end{matrix} \ \middle| \ bx \right) \\
&   \times\frac{{\phi}^{k - 1} (-1)}{q^{k -1 }} \sum\limits_{y, x_1, x_2, \ldots, x_{k-2}}  \phi(-1)\phi(1-y)\phi(x_1) \phi(x_2)\cdots \phi(x_{k-2}) {\overline{\phi}} \epsilon (1 - x_1)\cdots {\overline{\phi}} \epsilon (1 - x_{k-2}) \overline{ \phi }( 1- b y^{-1} x_1 \cdots x_{k-2})\\
&= \frac{\phi(-1)^k}{q} \sum\limits_{b} \phi(b)  \setlength\arraycolsep{1.5pt} {}_{n+1-k} F_{n - k} \left(\begin{matrix} A_0 & \ \ A_1 & \ \ \cdots & \ \ A_{n - k} \\ & \ \ B_1 & \ \ \cdots & \ \ B_{n - k} \end{matrix} \ \middle| \ bx \right)\\
&\times   \frac{{\phi}^{k - 1} (-1)}{q^{k -1 }} \sum\limits_{y \rightarrow y^{-1}, x_1, x_2, \ldots, x_{k-2}}  \phi(-1)\phi(1-y^{-1})\phi(x_1) \phi(x_2)\cdots \phi(x_{k-2}) {\overline{\phi}} \epsilon (1 - x_1)\cdots {\overline{\phi}} \epsilon (1 - x_{k-2}) \overline{ \phi }( 1- by x_1 \cdots x_{k-2})\\
&= \frac{\phi(-1)^k}{q} \sum\limits_{b} \phi(b)  \setlength\arraycolsep{1.5pt} {}_{n+1-k} F_{n - k} \left(\begin{matrix} A_0 & \ \ A_1 & \ \ \cdots & \ \ A_{n - k} \\ & \ \ B_1 & \ \ \cdots & \ \ B_{n - k} \end{matrix} \ \middle| \ bx \right) \\
&\times   \frac{{\phi}^{k - 1} (-1)}{q^{k -1 }} \sum\limits_{x_1, x_2, \ldots, x_{k-2},y}  \phi(x_1) \phi(x_2)\cdots \phi(x_{k-2}) \phi(y){\overline{\phi}} \epsilon (1 - x_1)\cdots {\overline{\phi}} \epsilon (1 - x_{k-2})\overline{\phi}\epsilon(1-y) \overline{ \phi }( 1- b x_1 \cdots x_{k-2}y)\\
&= \frac{\phi(-1)^k}{q} \sum\limits_{b} \phi(b)  \setlength\arraycolsep{1.5pt} {}_{n+1-k} F_{n - k} \left(\begin{matrix} A_0 & \ \ A_1 & \ \ \cdots & \ \ A_{n - k} \\ & \ \ B_1 & \ \ \cdots & \ \ B_{n - k} \end{matrix} \ \middle| \ bx \right) {}_{k} F_{k-1} \left(\begin{matrix} \phi & \ \ \phi \cdots & \ \ \phi \\ & \ \ \epsilon \cdots & \ \ \epsilon  \end{matrix} \ \middle| \ b\right).
\end{align*}
\end{footnotesize}
\end{proof} 

As an application of the inductive representation in Theorem \ref{thm:general_k}, we generalize some of the product formulas in \cite{{Barman+Tripathi-1}}.  Lemma \ref{lem:BarmanTripathiLemma3.2} may be generalized to express the product of Gaussian hypergeometric series of types $_2F_1$ and $_{n+1}F_n$ for any $n\ge 2$. 

\begin{theorem}\label{thm:product}
Let $ n\ge 2$ be an integer. Let $ A_0, A_1, \ldots, A_{n-2}$ and $ B_1, \ldots, B_{n-2}$ be multiplicative characters over $ \Fq$. Let $ x,z \in \Fq-\{0,1\}$.  One has
{\footnotesize\begin{align*}
 {}_{2} F_1 \left(\begin{matrix} \phi & \ \ \phi \\& \ \ \epsilon  \end{matrix} \ \middle| \ z \right){}_{n + 1} F_n \left(\begin{matrix} A_0 & \ \ A_1 &  \cdots & \ \ A_{n-2} &\phi &  \phi\\& \ \ B_1 &  \cdots & B_{n-2}& \epsilon&   \epsilon \end{matrix} \ \middle| \ x \right) 
 =  \frac{\phi(-1) }{q^2} \phi(1-z)\,\setlength\arraycolsep{1.5pt}{}_{n-1} F_{n - 2} \left(\begin{matrix} A_0 & \ \ A_1 & \ \ \cdots & \ \ A_{n - 2} \\ & \ \ B_1 & \ \ \cdots & \ \ B_{n - 2} \end{matrix} \ \middle| \ (1-z)x \right)\\+ \frac{1}{q}{}_{n-1} F_{n - 2} \left(\begin{matrix} A_0 & \ \ A_1 & \ \ \cdots & \ \ A_{n - 2} \\ & \ \ B_1 & \ \ \cdots & \ \ B_{n - 2} \end{matrix} \ \middle| \ x \right){}_{2} F_{1} \left(\begin{matrix} \phi &  \phi \\ & \ \ \epsilon  \end{matrix} \ \middle| \ 1\right){}_{2} F_{1} \left(\begin{matrix} \phi &  \phi \\ & \ \ \epsilon  \end{matrix} \ \middle| \ z\right)\\
 +\frac{ 1}{q^3} \sum_{w\in\Fq-\{1\}}\phi(w)\,\setlength\arraycolsep{1.5pt}{}_{n-1} F_{n - 2} \left(\begin{matrix} A_0 & \ \ A_1 & \ \ \cdots & \ \ A_{n - 2} \\ & \ \ B_1 & \ \ \cdots & \ \ B_{n - 2} \end{matrix} \ \middle| \ wx \right) F_4 (\phi;\phi;\epsilon;\epsilon;z(1-w);w(1 - z))^*.
\end{align*}}
\end{theorem}
\begin{proof}
The following identities hold because of Theorem \ref{thm:general_k}
{\footnotesize
\begin{align*}
 {}_{2} F_1 \left(\begin{matrix} \phi & \ \ \phi \\& \ \ \epsilon  \end{matrix} \ \middle| \ z \right){}_{n + 1} F_n \left(\begin{matrix} A_0 & \ \ A_1 &  \cdots & \ \ A_{n-2} &\phi &    \phi\\& \ \ B_1 &  \cdots & B_{n-2}& \epsilon& \epsilon \end{matrix} \ \middle| \ x \right) \\=
 \frac{ 1}{q} \sum_{w\in\Fq}\phi(w)\,\setlength\arraycolsep{1.5pt}{}_{n-1} F_{n - 2} \left(\begin{matrix} A_0 & \ \ A_1 & \ \ \cdots & \ \ A_{n - 2} \\ & \ \ B_1 & \ \ \cdots & \ \ B_{n - 2} \end{matrix} \ \middle| \ wx \right){}_{2} F_{1} \left(\begin{matrix} \phi &  \phi \\ & \ \ \epsilon  \end{matrix} \ \middle| \ w\right){}_{2} F_{1} \left(\begin{matrix} \phi &  \phi \\ & \ \ \epsilon  \end{matrix} \ \middle| \ z\right)\\
 = \frac{ 1}{q} \sum_{w\in\Fq-\{1\}}\phi(w)\,\setlength\arraycolsep{1.5pt}{}_{n-1} F_{n - 2} \left(\begin{matrix} A_0 & \ \ A_1 & \ \ \cdots & \ \ A_{n - 2} \\ & \ \ B_1 & \ \ \cdots & \ \ B_{n - 2} \end{matrix} \ \middle| \ wx \right)\Big(  \frac{\phi(-1) g(\epsilon)^2 }{q g(\phi)^2} F_4 (\phi;\phi;\epsilon;\epsilon;z(1-w);w(1 - z))^*\\ +  \frac{q \phi(-1) \phi(1 - z) \phi(w) \epsilon(1 - w)}{g(\phi)^4 } \delta\left(\frac{1 - w - z}{(1 - z)(1 - w)}\right)\Big)\\+ \frac{1}{q}{}_{n-1} F_{n - 2} \left(\begin{matrix} A_0 & \ \ A_1 & \ \ \cdots & \ \ A_{n - 2} \\ & \ \ B_1 & \ \ \cdots & \ \ B_{n - 2} \end{matrix} \ \middle| \ x \right){}_{2} F_{1} \left(\begin{matrix} \phi &  \phi \\ & \ \ \epsilon  \end{matrix} \ \middle| \ 1\right){}_{2} F_{1} \left(\begin{matrix} \phi &  \phi \\ & \ \ \epsilon  \end{matrix} \ \middle| \ z\right).
\end{align*}}
Since $g(\epsilon)= -1$ and $g(\phi)^2=\phi(-1)q$, the result follows. 
\end{proof}

\section{Evaluation of moments of Gaussian hypergeometric functions}
\label{sec:Moments}

From now on, we always write $_{n+1}F_n(x)$ to mean $ {}_{n+1} F_{n} \left( \begin{matrix}
\phi & \phi & \cdots & \phi \\
& \epsilon & \cdots & \epsilon 
\end{matrix} \middle |  \ x \right) $.

In this section, we show how to use the inductive representation introduced in \S\ref{sec:Induction} to evaluate first and second moments of Gaussian hypergeometric functions over $\Fq$. In \cite{Ono+Saad+Saikia}, explicit representations of moments of the hypergeometric functions $_2F_1$ and $_3F_2$ were given. Namely, the sums of the form $\sum_x {}_nF_{n-1}(x)^m$, $m\ge 1$, $n=2,3$, were expressed in terms of sums of class numbers.  Most of the moments we will evaluate in this section are defined as follows. 

\begin{definition} 
\label{Def:weighted}
For any positive integer $m $, the {\em $m$-th weighted-moment of the Gaussian hypergeometric function} $_{n+1}F_n$ is
$$\sum_ { x \in \mathbb{F}_ q} \phi(x) \cdot {}_{n+1} F_n(x)^m \ .$$
\end{definition}
We remark that the terms of the sums we propose in Definition \ref{Def:weighted} are weighted by the value of the quadratic character $\phi$.

\subsection{Traces of elliptic curves and moments}

In connection with finding values of hypergeometric functions $_2F_1(x)$ and $_3F_2(x)$, two families of elliptic curves are studied.  
The family of Legendre elliptic curves defined by $$E_ \lambda : y^2 =  x(x - 1)(x - \lambda),\qquad \lambda\in\Fq-\{0,1\}$$ and  the family of Clausen elliptic curves defined by  
\begin{equation*}\label{Clausen}
E'_{\lambda}: y^2 = (x-1) (x^2 + \lambda),\qquad\lambda\in\Fq-\{0,-1\}.
\end{equation*} 
 
Setting $a_{\lambda}(q)$ and $a'_{\lambda}(q)$ to be the trace of the Frobenius of $E_{\lambda} $ and $E'_{\lambda}$ over $\Fq$, respectively, one has
\begin{equation*} \label{trace of leg and hypergeometric}
a_{\lambda}(q)=-q\,\phi(-1)\,{}_2 F_1 (\lambda),
\end{equation*}
see \cite[Th. 1]{Ono}. Similaly,
 $$a'_{\frac{\lambda}{1-\lambda}} (q)^2=q + q^2 \phi\left(1- \lambda \right) \cdot {}_3 F_2\left(\lambda \right), \mbox { for } \lambda \neq 0, 1 $$ see \cite{Guindy+Ono}.

\subsection{First moments}

We start with proving the following useful lemma.

\begin{lemma}
\label{Lem:1}
For multiplicative characters $ A_0, A_1, \ldots A_{n},\psi$ and $ B_1, \ldots, B_{n}$ over $ \Fq$ and $ x \in \Fq$, we have
\begin{equation*}
\setlength\arraycolsep{1pt} 
{}_{n + 2} F_{n+1} \left(\begin{matrix} A_0 & \ \ A_1 & \ \ \cdots & \ \ A_{n} \ \ \psi\\ & \ \ B_1 & \ \ \cdots & \ \ B_{n} \ \ \psi\end{matrix} \ \middle| \ x \right)  =\frac{-1}{q } {}_{n + 1} F_{n} \left(\begin{matrix} A_0 & \ \ A_1 & \ \ \cdots & \ \ A_{n}\\ & \ \ B_1 & \ \ \cdots & \ \ B_{n} \end{matrix} \ \middle| \ x \right)+\begin{pmatrix}
A_0 \overline{\psi} \\ \overline{\psi}
\end{pmatrix} \cdots \begin{pmatrix}
A_{n}\overline{\psi}\\ B_{n} \overline{\psi}
\end{pmatrix}\overline{\psi}(x).
\end{equation*}
\end{lemma}
\begin{proof}
Using Definition (\ref{eq:defnhypergeometric}), one gets 
\begin{equation*}
\setlength\arraycolsep{1pt} 
{}_{n + 2} F_{n+1} \left(\begin{matrix} A_0 & \ \ A_1 & \ \ \cdots & \ \ A_{n} \ \ \psi\\ & \ \ B_1 & \ \ \cdots & \ \ B_{n} \ \ \psi\end{matrix} \ \middle| \ x \right)  = \frac{q}{q - 1} \sum\limits_{\chi} \begin{pmatrix}
A_0 \chi \\ \chi 
\end{pmatrix} \cdots \begin{pmatrix}
A_{n} \chi \\ B_{n} \chi
\end{pmatrix}\begin{pmatrix}
\psi \chi \\ \psi\chi
\end{pmatrix} \chi(x).
\end{equation*}
One knows that $\begin{pmatrix}
\psi \chi \\ \psi\chi
\end{pmatrix}=-\frac{1}{q} + \frac{q-1}{q}\delta(\psi\chi) $ where $\delta(\chi)=0$ if $\chi\ne\epsilon$ and $\delta(\epsilon)=1$, \cite[Equation (2.12)]{Greene}. Thus,  the statement holds by definition (\ref{eq:defnhypergeometric}).
\end{proof}

In the following proposition, we compute the first moment and weighted moment of the Gaussian hypergeomtric function $_{n+1}F_n$. 

\begin{proposition}
\label{prop:1}
Let $n\ge 1$ be an integer. The following identity holds true. 
\begin{equation*}
\setlength\arraycolsep{1pt} 
q^n\, \sum\limits_{y} \setlength\arraycolsep{1pt} {}_{n+1} F_{n } (y )=\left(-1\right)^{n+1} .
\end{equation*}
\end{proposition}
\begin{proof}
Setting $A_{n+1}=B_{n+1}=\epsilon$ in Theorem \ref{th:GreensTh3.13}, for any $x\in\mathbb{F}_q^{\times}$ one has 
\begin{equation*}
\setlength\arraycolsep{1pt} 
{}_{n + 2} F_{n+1} \left(\begin{matrix} A_0 & \ \ A_1 & \ \ \cdots & \ \ A_{n}\ \ \epsilon \\ & \ \ B_1 & \ \ \cdots & \ \ B_{n}\ \  \epsilon \end{matrix} \ \middle| \ x \right) = \frac{1}{q} \sum\limits_{y\ne0,1} \setlength\arraycolsep{1pt} {}_{n+1} F_{n } \left(\begin{matrix} A_0 & \ \ A_1 & \ \ \cdots & \ \ A_{n } \\ & \ \ B_1 & \ \ \cdots & \ \ B_{n } \end{matrix} \ \middle| \ xy \right).
\end{equation*}
Using Lemma \ref{Lem:1} where we set $A_0=\cdots=A_{n}=\phi$, $B_1=\cdots= B_n=\psi=\epsilon$,  it follows that 
\begin{eqnarray*}
\setlength\arraycolsep{1pt} 
{}_{n + 2} F_{n+1} \left(\begin{matrix} \phi & \ \ \phi & \ \ \cdots & \ \ \phi \ \ \epsilon\\ & \ \ \epsilon & \ \ \cdots & \ \ \epsilon \ \ \epsilon\end{matrix} \ \middle| \ x \right) = -\frac{1}{q}\, {}_{n+1}F_n ( x)  +\begin{pmatrix}
\phi\\ \epsilon
\end{pmatrix}^{n+1} \epsilon(x).
\end{eqnarray*}

Recall that $\begin{pmatrix}
\chi \\ \chi 
\end{pmatrix}=\begin{pmatrix}
\chi \\ \epsilon
\end{pmatrix}=-\frac{1}{q}+\frac{q-1}{q}\delta(\chi)$ for a multiplicative character $\chi$ where $\delta(\chi)=0$ if $\chi\ne\epsilon$ and $\delta(\epsilon)=1$, \cite[Equation (2.12)]{Greene}. One then gets
\begin{equation*}
\setlength\arraycolsep{1pt} 
 \sum\limits_{y\ne1} \setlength\arraycolsep{1pt} {}_{n+1} F_{n } (xy )=-{}_{n+1}F_n (x )  -\left(-\frac{1}{q}\right)^{n},
\end{equation*}
hence the statement of the proposition follows.
\end{proof}
\begin{remark}
Proposition \ref{prop:1} was proved for $n=1$ and $q\equiv 3\pmod 4$ in \cite[Proposition 2.11 (3)]{Ono+Saad+Saikia}. 
When $ q \equiv 1 \pmod 4$, comparing Proposition \ref{prop:1} to \cite[Proposition 2.11 (4)]{Ono+Saad+Saikia}, we  observe that    
$$ \sum_{ \substack{\gcd(s, q) = 1 \\ s \equiv q + 1 \pmod 8}} H ^* \left(\frac{4q-s^2}{4} \right) s + 2 \sum_{ \substack{ \gcd(s, q) = 1 \\ s \equiv q + 1 \pmod {16}}} H ^* \left(\frac{4q-s^2}{16} \right) s= -1,$$ 
 where $s$ runs over all integers in the interval $-2 \sqrt{q} \leq s \leq 2 \sqrt{q}$ and $ H^*( \bullet)$ is the Hurwitz class number as defined in \cite{eicher}.
\end{remark}

In the following proposition, we find an expression for the first weighted-moment of the Gaussian hypergeometric function $_{n+1}F_n$.
\begin{proposition}
\label{prop:2}
Let $n\ge 1$ be an integer. The following identity holds true. 
\begin{equation*}
\setlength\arraycolsep{1pt} 
q^n\, \sum\limits_{y}\phi(y) \cdot \setlength\arraycolsep{1.5pt} {}_{n+1} F_{n }( y)=\left(-\phi(-1)\right)^{n+1}.
\end{equation*}
\end{proposition}
\begin{proof}
Setting $A_0=\cdots=A_{n}=A_{n+1}=\phi$, $B_1=\cdots= B_n=\epsilon$ and $B_{n+1}=\phi$ in Theorem \ref{th:GreensTh3.13}, it follows that for any $x\in\mathbb{F}_q^{\times}$ one has 
\begin{equation*}
\setlength\arraycolsep{1pt} 
{}_{n + 2} F_{n+1} \left(\begin{matrix} \phi & \ \ \phi & \ \ \cdots & \ \ \phi\ \ \phi \\ & \ \ \epsilon & \ \ \cdots & \ \ \epsilon\ \  \phi \end{matrix} \ \middle| \ x \right) = \frac{1}{q} \sum\limits_{y\ne0,1}\phi(y) \setlength\arraycolsep{1.5pt} {}_{n+1} F_{n } (  xy ).
\end{equation*}
Using Lemma \ref{Lem:1}, one obtains
\begin{equation*}
\setlength\arraycolsep{1pt} 
{}_{n + 2} F_{n+1} \left(\begin{matrix} \phi & \ \ \phi & \ \ \cdots & \ \ \phi \ \ \phi\\ & \ \ \epsilon & \ \ \cdots & \ \ \epsilon \ \ \phi\end{matrix} \ \middle| \ x \right)  = -\frac{1}{q}\, {}_{n+1}F_n (x ) +\begin{pmatrix}
\epsilon \\ \phi
\end{pmatrix} ^{n+1} \phi(x),\,
\end{equation*}
where $\begin{pmatrix}
\epsilon \\ \phi
\end{pmatrix}=-\frac{\phi(-1)}{q}$, \cite[Equation (2.13)]{Greene}.
This implies that 
\begin{equation*}
\setlength\arraycolsep{1pt} 
 \sum\limits_{y\ne1}\phi(xy) \setlength\arraycolsep{1.5pt} {}_{n+1} F_{n } ( xy )=-\phi(x)\,{}_{n+1}F_n (x )  +\frac{(-\phi(-1))^{n+1}}{q^n},
\end{equation*}
hence the result.
\end{proof}

\begin{corollary}
\label{Cor:2}
The following three identities hold:
\begin{itemize}
\item[1)] $\sum_{\lambda\ne0,1} a_{\lambda}(q) +\phi(-1)=-1$
\item[2)]  $\sum_{\lambda\ne 0,-1}\phi(1+\lambda) a'_{\lambda}(q)^2 + q + q^2\,\, _3F_2(1)=-1$
\item[3)]  $\sum_{\lambda\ne 0,-1} \phi(\lambda) a'_{\lambda}(q)^2+q\,\phi(-1) + q^2\,\, _3F_2(1)=-\phi(-1)$.
\end{itemize}
\end{corollary}
\begin{proof}
Statement 1) follows from either Proposition \ref{prop:1} or \ref{prop:2} with $n=1$ and noticing that $_2F_1(1)=-\phi(-1)/q$, \cite[Proposition 3]{Ono}. Statement 2) holds by Proposition \ref{prop:1} for $n=2$ and applying the transformation $\lambda/(1+\lambda)\mapsto\lambda$. In a similar fashion, statement 3) follows from Proposition \ref{prop:2} and the fact that $\sum_{\lambda}\phi(\lambda(1+\lambda))=-1$.
\end{proof}

\subsection{Second weighted moments}
In the following theorem, we find an expression for the second weighted-moment of the Gaussian hypergeometric function $_{n+1}F_n$ as a sum of products of Gaussian hypergeometric functions of lower levels.
\begin{theorem}
\label{thm:1}
Let $ n$ and $k$ be integers with $ n > k \geq 1$.  Let $x\in\Fq$. The following identity holds
\begin{align}\label{n+1=2k}
{}_{n+1} F_{n} (x) = 
\frac{\phi(-1)^k}{q}\sum_{ \lambda \in \mathbb{F}_q } \phi (\lambda) \cdot {}_{n+1-k} F_{n-k} (\lambda x){}_{k} F_{k-1} (  \lambda ).
\end{align}  
In particular, for any positive integer $ k \geq 2$, the second weighted moment of the Gaussian hypergeometric series $_{2k}F_{2k-1}$ can be expressed as follows
$$q\, \phi(-1)^k \cdot {}_{2k} F_{2k-1} (1)=\sum_{ \lambda \in \mathbb{F}_q} \phi(\lambda) \cdot {}_{k} F_{k-1} (\lambda)^2.$$ 
\end{theorem}
\begin{proof}
The proof of the first statement follows by setting $ A_0 = A_1 = \cdots = A_{n-k} =\phi$ and $ B_1 =\cdots = B_ {n-k} = \epsilon$ in Theorem \ref{thm:general_k}, whereas the second statement is established by setting $n+1=2k$ and $x=1$.
\end{proof}

\begin{remark}
We note that instances of Theorem \ref{thm:1} have been proved in other articles. For example,  the second formula above has been proved for $k=2$ in \cite[Lemma 2.2]{Ono1} during the course of realizing the value $_4F_3(1)$ as a character sum using the fact that a certain Calabi-Yau threefold is modular. The first formula in Theorem \ref{thm:1} was proved  for $n=5$ and $k=2$ in \cite[Equation (5.28)]{Fre}, whereas the second identity was established when $k=3$ in \cite[Equation (5.26)]{Fre}.
\end{remark}

\begin{corollary}
\label{cor:estimates}
The following are true:
\begin{itemize}
\item[1)] $\left| _4F_3(1)-\frac{1}{q^3}\right|\le \frac{1}{q}+O\left( 1/q^{\frac{3}{2}}\right)$
\item[2)] $|_6F_5(1)|=O(1/q^2)$.
\end{itemize}
\end{corollary}
\begin{proof}
One knows that $q _4F_3(1)=\sum_{\lambda}\phi(\lambda) _2F_1(\lambda)^2$, see Theorem \ref{thm:1}. It follows that $$q _4F_3(1)=\sum_{\lambda}\phi(\lambda)(-\phi(-1)a_{\lambda}(q)/q)^2+ _2F_1(1)^2= \sum_{\lambda\ne 0,1}\phi(\lambda) a_{\lambda}(q)^2/q^2+1/q^2.$$
Thus 
$$\left| _4F_3(1)-\frac{1}{q^3} \right|\le \frac{1}{q^3}\sum_{\lambda\ne 0,1}a_{\lambda}(q)^2.$$
Michel proved in \cite{Mi} that for a family of one-parameter families of elliptic curves with a non-constant $j$-invariant $j(\lambda )$, one has
$$\sum_{\lambda\ne 0,1}a_{\lambda}(q)^2=q^2+O(q^{3/2}),$$
hence statement 1) follows.

Similarly,  \begin{eqnarray*}
q\phi(-1)\,  _6F_5(1)&=&\sum_{\lambda} \phi(\lambda) _3F_2(\lambda)^2\\
&=& \frac{ 1 }{q^4}\sum_{\lambda\ne 0,1}\phi(\lambda)\left(a'_{\frac{\lambda}{1-\lambda}} (q)^2-q\right)^2  + _3F_2(1)^2\\
&=&\frac{1}{q^4}\sum_{\lambda\ne 0,-1}\phi(\lambda(1+\lambda))a'_{\lambda}(q)^4-\frac{2}{q^3}\sum_{\lambda\ne 0,-1}\phi(\lambda(1+\lambda))a'_{\lambda}(q)^2+\frac{1}{q^2}\sum_{\lambda\ne 0,-1}\phi(\lambda(1+\lambda))\\&+&  _3F_2(1)^2.
\end{eqnarray*}
We now estimate the terms above. The sum $\sum_{\lambda\ne 0,-1}\phi(\lambda(1+ \lambda))=-1$. The result of \cite{Mi} implies that  $\sum_{\lambda\ne0,1}\phi(\lambda(1+\lambda))a'_{\lambda}(q)^2=q^2+O(q^{3/2})$ and $_3F_2(1)=O(1/q^{1/2})$, see Corollary \ref{Cor:2}.  In Lemma C.1 of \cite{Weng}, It was shown that for a family of one-parameter families of elliptic curves with a non-constant $j$-invariant $j(\lambda )$, one has $\sum_{\lambda\ne 0,1}a'_{\lambda}(q)^4=2q^3+O(q^{5/2}).$ Statement 2) now follows by comparing these estimates.
\end{proof}

\section{Generating functions of Gaussian hypergeometric functions}
\label{sec:GeneratingFunctions}

Using Definition (\ref{eq:defnhypergeometric}) and Theorem \ref{th:GreensTh3.13}, one obtains

\begin{footnotesize}
\begin{equation*}
\setlength\arraycolsep{1pt} 
{}_{n + 1} F_n \left(\begin{matrix} A_0 & \ \ A_1 & \ \ \cdots & \ \ A_n \\ & \ \ B_1 & \ \ \cdots & \ \ B_n \end{matrix} \ \middle| \ x \right) = \frac{A_nB_n(-1)}{q}\sum_y\frac{q}{q-1} \sum_{\chi}{A_0\chi\choose\chi}{A_1\chi\choose B_1\chi}\cdots {A_{n-1}\chi\choose B_{n-1}\chi}\chi(xy) A_n(y)\overline{A}_nB_n(1-y).
\end{equation*}
\end{footnotesize}

We now proceed towards proving the main theorem of this section.

\begin{footnotesize}
\begin{theorem}
\label{thm:gen}
For any integer $n\ge 1$,  $x\in\Fq^{\times}$ and $t\in\Fq-\{0,1\}$, the following is true:
\begin{align*}
\setlength\arraycolsep{1pt} 
\frac{q}{q-1}\sum_{\psi}{A_n\overline{B}_n\psi\choose\psi} {}_{n+1 } F_{n} \left(\begin{matrix} A_0 & \ \ A_1 & \ \ \cdots & \ \ A_{n-1}& \ \ A_n\psi \\ & \ \ B_1 & \ \ \cdots & \ \ B_{n-1} & \ \ B_n\end{matrix} \ \middle| x \right)\psi(t)={}_{n + 1} F_n \left(\begin{matrix} A_0 & \ \ A_1 & \ \ \cdots & \ \ A_n \\ & \ \ B_1 & \ \ \cdots & \ \ B_n \end{matrix} \ \middle| \frac{ x}{1-t} \right) \overline{A}_n(1-t)
\end{align*}
\end{theorem}

\begin{proof}
In what follows $v=y/(1-t)$.

\begin{align*}
\setlength\arraycolsep{1pt} 
{}_{n + 1} F_n \left(\begin{matrix} A_0 & \ \ A_1 & \ \ \cdots & \ \ A_n \\ & \ \ B_1 & \ \ \cdots & \ \ B_n \end{matrix} \ \middle| \frac{ x}{1-t} \right) = \frac{A_nB_n(-1)}{q}\sum_y\frac{q}{q-1} \sum_{\chi}{A_0\chi\choose\chi}\cdots {A_{n-1}\chi\choose B_{n-1}\chi}\chi\left(\frac{xy}{1-t}\right) A_n(y)\overline{A}_nB_n(1-y)\\
=\frac{A_nB_n(-1)}{q}\sum_v\frac{q}{q-1} \sum_{\chi}{A_0\chi\choose\chi}\cdots {A_{n-1}\chi\choose B_{n-1}\chi}\chi\left(xv\right) A_n(v(1-t))\overline{A}_nB_n(1-v(1-t))\\
= \frac{A_nB_n(-1)}{q}\sum_v\frac{q}{q-1} \sum_{\chi}{A_0\chi\choose\chi}\cdots {A_{n-1}\chi\choose B_{n-1}\chi}\chi\left(xv\right) A_n(v(1-t))\overline{A}_nB_n(1-v)\overline{A}_nB_n\left(1+\frac{vt}{1-v}\right)\\
=\frac{A_nB_n(-1)}{q}\sum_v\frac{q^2}{(q-1)^2} \sum_{\chi,\psi}{A_0\chi\choose\chi}\cdots {A_{n-1}\chi\choose B_{n-1}\chi}\chi\left(xv\right) A_n(v(1-t))\overline{A}_nB_n(1-v){A_n\overline{B}_n\psi\choose\psi}\psi\left(\frac{-vt}{1-v}\right)\\
=\frac{qA_nB_n(-1)}{(q-1)^2}\sum_v \sum_{\chi,\psi}{A_0\chi\choose\chi}\cdots {A_{n-1}\chi\choose B_{n-1}\chi}{A_n\overline{B}_n\psi\choose\psi}\chi\left(xv\right) A_n(v(1-t))\overline{A_n\psi}B_n(1-v)\psi(-vt)\\
=\frac{ A_nB_n(-1)}{q-1}\sum_v\frac{q}{q-1}A_n(v(1-t))\sum_{\psi}{A_n\overline{B}_n\psi\choose\psi} \overline{A_n\psi}B_n(1-v) \psi(-vt)\sum_{\chi}{A_0\chi\choose\chi}\cdots {A_{n-1}\chi\choose B_{n-1}\chi}\chi\left(xv\right)\\
=\frac{ A_nB_n(-1)}{q-1}\sum_vA_n(v(1-t))\sum_{\psi}{A_n\overline{B}_n\psi\choose\psi} \overline{A_n\psi}B_n(1-v) \psi(-vt){}_{n } F_{n-1} \left(\begin{matrix} A_0 & \ \ A_1 & \ \ \cdots & \ \ A_{n-1} \\ & \ \ B_1 & \ \ \cdots & \ \ B_{n-1} \end{matrix} \ \middle| xv \right)\\
=\frac{ A_n\psi B_n(-1)}{q-1}\sum_{\psi}A_n(1-t)\psi(t){A_n\overline{B}_n\psi\choose\psi} \sum_vA_n\psi(v)\overline{A_n\psi}B_n(1-v) {}_{n } F_{n-1} \left(\begin{matrix} A_0 & \ \ A_1 & \ \ \cdots & \ \ A_{n-1} \\ & \ \ B_1 & \ \ \cdots & \ \ B_{n-1} \end{matrix} \ \middle| xv \right)\\
=\frac{q}{q-1}\sum_{\psi}A_n(1-t)\psi(t){A_n\overline{B}_n\psi\choose\psi} {}_{n+1} F_{n} \left(\begin{matrix} A_0 & \ \ A_1 & \ \ \cdots & \ \ A_{n-1}&A_n\psi \\ & \ \ B_1 & \ \ \cdots & \ \ B_{n-1} &B_n\end{matrix} \ \middle| x \right).
\end{align*}
where we used the identity
\begin{equation*}
\overline{A}_nB_n\left(1+\frac{vt}{1-v}\right)=\delta\left(-\frac{vt}{1-v}\right)+\frac{q}{q-1}\sum_{\psi}{A_n\overline{B}_n\psi\choose\psi}\psi\left(\frac{-vt}{1-v}\right),
\end{equation*}
see \cite[Equation (2.10)]{Greene}.
\end{proof}
\end{footnotesize}

We notice that in the theorem above the position of $\psi$ can be changed to be multiplied times any of the $A_i$, $1\le i \le n$.

The following corollary provides a closed form solution for a weighted-sum of certain Gaussian hypergeometric series over all the characters over $\Fq$.
\begin{corollary}
\label{cor:gen}
Let $A\ne\epsilon$ be a multiplicative character over $\Fq$. For any integer $n\ge 1$,  $x\in\Fq^{\times}$ and $t\in\Fq-\{0,1\}$, one has
\begin{eqnarray*}
\setlength\arraycolsep{1pt} 
q\sum_{\psi\ne \epsilon} {}_{n+2 } F_{n+1} \left(\begin{matrix} A & \ \ A & \ \ \cdots & \ \ A& \ \ \psi \\ & \ \ \epsilon& \ \ \cdots & \ \ \epsilon & \ \ \epsilon\end{matrix} \ \middle| x \right)\psi(t)&=&\\ (q-1){}_{n + 1} F_n \left(\begin{matrix} A & \ \ A & \ \ \cdots & \ \ A \\ & \ \ \epsilon & \ \ \cdots & \ \ \epsilon \end{matrix} \ \middle| \frac{ x}{1-t} \right)&-&(q-2){}_{n + 1} F_n \left(\begin{matrix} A & \ \ A & \ \ \cdots & \ \ A \\ & \ \ \epsilon & \ \ \cdots & \ \ \epsilon \end{matrix} \ \middle| x \right)+\left(\frac{-1}{q}\right)^n.
\end{eqnarray*}
\end{corollary}
\begin{proof}
In view of Lemma \ref{Lem:1}, one sees that 
\begin{align*}
\setlength\arraycolsep{1pt} 
{}_{n + 2} F_{n+1} \left(\begin{matrix} A & \ \ A & \ \ \cdots & \ \ A \ \ \epsilon \\ & \ \ \epsilon & \ \ \cdots & \ \ \epsilon \ \ \epsilon \end{matrix} \ \middle| \frac{ x}{1-t} \right)=\frac{-1}{q}{}_{n + 1} F_{n} \left(\begin{matrix} A & \ \ A & \ \ \cdots & \ \ A  \\ & \ \ \epsilon & \ \ \cdots & \ \ \epsilon  \end{matrix} \ \middle| \frac{ x}{1-t} \right)+\begin{pmatrix}
A \\ \epsilon
\end{pmatrix}^{n+1},\quad \begin{pmatrix}
A \\ \epsilon
\end{pmatrix}=-\frac{1}{q}.
\end{align*}
In accordance with Theorem \ref{thm:gen}, one notices that
\begin{eqnarray*}
\setlength\arraycolsep{1pt} 
{}_{n + 2} F_{n+1} \left(\begin{matrix} A & \ \ A & \ \ \cdots & \ \ A \ \ \epsilon \\ & \ \ \epsilon & \ \ \cdots & \ \ \epsilon \ \ \epsilon \end{matrix} \ \middle| \frac{ x}{1-t} \right)&=&
\frac{q}{q-1}\sum_{\psi}{\psi\choose\psi} {}_{n+2 } F_{n+1} \left(\begin{matrix} A & \ \ A & \ \ \cdots & \ \ A& \ \ \psi \\ & \ \ \epsilon& \ \ \cdots & \ \ \epsilon & \ \ \epsilon\end{matrix} \ \middle| x \right)\psi(t).
\end{eqnarray*}
Again Lemma \ref{Lem:1} implies that
\begin{eqnarray*}
\setlength\arraycolsep{1pt} 
{}_{n + 2} F_{n+1} \left(\begin{matrix} A & \ \ A & \ \ \cdots & \ \ A \ \ \epsilon \\ & \ \ \epsilon & \ \ \cdots & \ \ \epsilon \ \ \epsilon \end{matrix} \ \middle|  x \right)&=&
\frac{-1}{q}{}_{n + 1} F_{n} \left(\begin{matrix} A & \ \ A & \ \ \cdots & \ \ A  \\ & \ \ \epsilon & \ \ \cdots & \ \ \epsilon  \end{matrix} \ \middle| x \right)+\left(\frac{-1}{q}\right)^{n+1},
\end{eqnarray*}
hence, the fact that $\begin{pmatrix}
\psi\\ \psi
\end{pmatrix}=-\frac{1}{q}$ when $\psi\ne\epsilon$; whereas $\begin{pmatrix}
\epsilon \\ \epsilon
\end{pmatrix}=\frac{q-2}{q}$ yields that
\begin{eqnarray*}
\setlength\arraycolsep{1pt} 
{}_{n + 2} F_{n+1} \left(\begin{matrix} A & \ \ A & \ \ \cdots & \ \ A \ \ \epsilon \\ & \ \ \epsilon & \ \ \cdots & \ \ \epsilon \ \ \epsilon \end{matrix} \ \middle| \frac{ x}{1-t} \right)
&=& \frac{-q}{q(q-1)}\sum_{\psi\ne \epsilon} {}_{n+2 } F_{n+1} \left(\begin{matrix} A & \ \ A & \ \ \cdots & \ \ A& \ \ \psi \\ & \ \ \epsilon& \ \ \cdots & \ \ \epsilon & \ \ \epsilon\end{matrix} \ \middle| x \right)\psi(t)\\&+&\frac{q(q-2)}{q(q-1)}\left( \frac{-1}{q}{}_{n + 1} F_{n} \left(\begin{matrix} A & \ \ A & \ \ \cdots & \ \ A  \\ & \ \ \epsilon & \ \ \cdots & \ \ \epsilon  \end{matrix} \ \middle| x \right)+\left(\frac{-1}{q}\right)^{n+1}\right),
\end{eqnarray*}
hence the result.
\end{proof}

\begin{remark}
In view of Corollary \ref{cor:gen}, one can see for example that 

\begin{eqnarray*}
\setlength\arraycolsep{1pt} 
\sum_{\psi\ne \epsilon} {}_{3 } F_{2} \left(\begin{matrix} \phi & \ \ \phi & \ \ \psi \\ & \ \ \epsilon & \ \ \epsilon\end{matrix} \ \middle| \lambda \right)\psi(t)=\frac{q-1}{q} {}_{2} F_1 \left(\begin{matrix} \phi & \ \ \phi  \\ & \ \ \epsilon  \end{matrix} \ \middle| \frac{ \lambda}{1-t} \right)-\frac{q-2}{q}{}_{2} F_1 \left(\begin{matrix} \phi & \ \ \phi  \\ & \ \ \epsilon  \end{matrix} \ \middle| \lambda \right)-\frac{1}{q^2}.
\end{eqnarray*}
In particular, if $t=1-\lambda^2$, then 
\begin{eqnarray*}
\setlength\arraycolsep{1pt} 
\sum_{\psi\ne \epsilon} {}_{3 } F_{2} \left(\begin{matrix} \phi & \ \ \phi & \ \ \psi \\ & \ \ \epsilon & \ \ \epsilon\end{matrix} \ \middle| \lambda \right)\psi(1-\lambda^2)=\left(\frac{q-1}{q}\phi(\lambda)-\frac{q-2}{q}\right){}_{2} F_1( \lambda )-\frac{1}{q^2},
\end{eqnarray*}
where $_2F_1(\lambda)=-\phi(-1)a_{\lambda}(q)/q$. In a similar fashion, 
\begin{eqnarray*}
\setlength\arraycolsep{1pt} 
\sum_{\psi\ne \epsilon} {}_{4 } F_{3} \left(\begin{matrix} \phi & \ \ \phi &\ \ \phi& \ \ \psi \\ & \ \ \epsilon & \ \ \epsilon&\ \ \epsilon\end{matrix} \ \middle| \lambda \right)\psi(1-\lambda^2)=\left(\frac{q-1}{q}\phi(-\lambda)-\frac{q-2}{q}\right){}_{3} F_2( \lambda )+\frac{1}{q^3}.
\end{eqnarray*}
where $_3F_2(\lambda)$ is the trace of Frobenius of the Clausen elliptic curve $E'$, see \S \ref{sec:Moments}.
\end{remark}

\end{document}